\documentclass[12pt]{amsart}

\usepackage{fullpage, graphicx}

\newtheorem{thm}{Theorem}[section]
\newtheorem{prop}[thm]{Proposition}

\theoremstyle{definition}
\newtheorem{example}[thm]{Example}

\title{Tree hook length formulae, Feynman rules and B-series}

\author{Bradley R. Jones and Karen Yeats}

\begin{document}
\begin{abstract}
We consider weighted generating functions of trees where the weights are products of functions of the sizes of the subtrees. 
This work begins with the observation that three different communities, largely independently, found substantially the same result concerning these series.  We unify these results with a common generalization.
Next we use the insights of one community on the problems of another in two different ways.
Namely, we use the differential equation perspective to find a number of new interesting hook length formulae for trees, and we use the body of examples developed by the combinatorial community to give quantum field theory toy examples with nice properties.
\end{abstract}

\maketitle

\section{Introduction}

The tree factorial, $t!$, for a rooted tree, $t$, is the product of the sizes of the subtrees of $t$.  For example 
\[
\vcenter{\hbox{\includegraphics{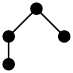}}}! = 1\cdot2\cdot1\cdot 4 = 8.
\]  The tree factorial is an elegant and classical function on rooted trees.  Three different communities each working with rooted trees for different reasons generalized this simple example in their own language and from their own perspective.  The main results of each of these generalizations are equivalent over their common hypotheses.

\medskip

The first community is the enumerative combinatorics community. 
For the tree factorial itself, Knuth (\cite{Kn73} p70) gave as an exercise to show that $\frac{|t|!}{t!}$ counts the number of ways to label a plane tree, $t$, with increasing labels. The enumerative combinatorics perspective is to use generalizations of the tree factorial to produce equations which equate a power series, called a hook length series, to the generating function of a combinatorial class. These equations, called hook length formulae, are of importance in combinatorics as they often imply bijections between combinatorial classes.

An early example of a hook length formula was given by Postnikov in 2004 \cite{Po04}:
\begin{equation}
\label{eq of Postnikov}
\sum_{t \in \mathcal{B}_n}n!\prod_{v \in V(t)}\left(1 + \frac{1}{|t_v|}\right) = 2^n(n + 1)^{n-1}.
\end{equation}
The left hand side of the equation is a hook length series that also counts the number of bicoloured binary plane trees with a particular labelling. The right side of the equation counts the number of bicoloured labeled forests. In 2005, Seo \cite{Se05} developed a bijection between these two combinatorial classes.  Other combinatorial works on hook length formulae include: \cite{BW89,Sa01,GS06,CY08,Ha08,Ya08,CGG09,SZ09,Ha10,CGG13,FG13,KP12b}.

Kuba and Panholzer \cite{KP12} discovered a general identity of hook length series in the form of a recurrence relation on the coefficients of the hook length series.  More recently they extended their results to a study of multilabelled increasing trees \cite{KP14}.

\medskip

The second community is the B-series community.  B-series are power series solutions of differential equations indexed by trees which were originally developed in the analysis of Runge-Kutta methods \cite{Bu63}.  The tree factorial is used as a statistic in the analysis of Runge-Kutta methods for computing approximate solutions of differential equations \cite{HNW93}.  Mazza in \cite{Ma04} gives a theorem concerning such B-series solutions which is equivalent to the Kuba and Panholzer recurrence.

\medskip

The third community is the community which takes a Hopf algebraic approach to renormalization in quantum field theory.  This approach began with the work of Connes and Kreimer \cite{CK98}.  The underlying algebraic structure here is the same as for B-series, a fact which was recognized by Brouder in 2000 \cite{Br00}.  
$\frac{1}{t!}$ defines the simplest non-trivial Feynman rules for rooted trees and is the leading term more generally, see Panzer \cite[p.~38]{Pa11}.  Panzer \cite{Pa11} extended this to a broader understanding of how the algebraic structure and the Feynman rules interact.  His Feynman rules on trees function as hook weights and so again hook length formulae appear.

\medskip

Each community thus has a differently flavoured and independently achieved, perspective on the main result.  From an enumerative combinatorics perspective the result is given in terms of coefficient extraction, from the B-series perspective the result is given in terms of a differential equation, and from the combinatorial Hopf algebra and quantum field theory perspective the result is given in terms of an integral equation and a universal property.
We bring the results of all three communities together into a common language, explaining their set ups and giving a common generalization.

We then look at two ways in which we can use the insights of one community to throw light on the questions of another.  
First, in section \ref{sec using Mazza} we look at using the differential equation formulation to obtain new combinatorial insights.  Specifically, we develop new methods to apply the differential equation in cases where the hook weights are not nice functions in the sense that they are either piecewise or their growth is too fast.  We then give a table of new hook length formulae, some found using these methods and some found with existing methods. 
Second, in section \ref{sec qft models} we use the many examples from the enumerative combinatorics to obtain interesting toy models for quantum field theory.  Given a hook length formula, the translation into quantum field theory language is as follows: the tree class used determines the Dyson-Schwinger equation; the  hook weight determines the Feynman rules; the hook length formula itself gives a nice form for the Green function.

The first four sections consist of results from the first author's MSc thesis \cite{Jo14}.

\section{Background and notation}

We will follow the notation of Flajolet and Sedgewick \cite{FS09} for combinatorial specifications and generating functions.  Combinatorial classes will be given script letters, e.g.~$\mathcal{C}$, with the generating function of the class being given the associated roman letter, e.g.~$C(z)$, except as otherwise specified.  $\mathcal{C}_n$ denotes those elements of $\mathcal{C}$ of size $n$, and generating functions are ordinary in the unlabelled case $C(z) = \sum_{n\geq 0}|\mathcal{C}_n|z^n$ and exponential in the labelled case $C(z) = \sum_{n\geq 0}|\mathcal{C}_n|z^n/n!$.  We use standard combinatorial operators including $\times$ for cartesian product, $\star$ for labelled product, and \textsc{seq} for the sequence operator.

Here we are primarily interested in combinatorial classes of trees. An unlabelled (labelled) class of trees, $\mathcal{T}$, is \emph{simple} if there exists a combinatorial operator, $\Phi$, and a size preserving bjiection, $\gamma: \mathcal{T} \to \mathcal{Z} \times \Phi(\mathcal{T})$, ($\gamma: \mathcal{T} \to \mathcal{Z} \star \Phi(\mathcal{T})$), such that for all $x \in \Phi(\mathcal{T})$, $x = \{t_1, \ldots, t_k\}$ for some $t_1, \ldots, t_k \in \mathcal{T}$ and for $t \in \mathcal{T}$, $\gamma(t) = (\bullet, \{t_1, \ldots, t_k\})$ if and only if $t$ is a tree where $t_1, \ldots, t_j$ are the subtrees of $t$ whose roots are the children of the root of $t$. Thus a simple tree is one where every vertex has a $\Phi$-structure of children.  Write $\phi$ for the power series corresponding to the combinatorial operator $\Phi$.

For a forest, $f$, and a vertex, $v \in V(f)$, we denote as $f_v$ the subtree of $f$ whose root is $v$. For a subset of vertices, $W \subset V(f)$, we denote $f_W$ to be the forest with trees, $f_v$ for $v \in W$.

In order to define the main theorem and use its applications to prove hook length formulae, we shall define decorated trees. 
A \emph{decorated tree} is a rooted tree where each vertex is given a positive integer size. The size of a decorated tree is the sum of the sizes of it vertices.

We can define simple classes of decorated trees similarly to simple classes of ordinary trees. We say a unlabelled (labelled) class of decorated trees, $\mathcal{T}'$, is \emph{simple} if there exists a bivariate combinatorial operator, $\Phi$, and a size preserving bjiection, $\gamma: \mathcal{T}' \to \mathcal{Z} \times \Phi(\mathcal{Z}, \mathcal{T}')$, ($\gamma: \mathcal{T}' \to (\mathcal{Z} \star \Phi(\mathcal{Z}, \mathcal{T}')$), such that for all $x \in \Phi(\mathcal{Z}, \mathcal{T}')$, $x = (\bullet^i, \{t_1, \ldots, t_k\})$ for some $t_1, \ldots, t_k \in \mathcal{T}'$ and $i \in \mathbb{N}$ and for $t \in \mathcal{T}'$, $\gamma(t) = (\bullet, (\bullet^{i-1}, \{t_1, \ldots, t_k\}))$ if and only if $t$ is a tree where $t_1, \ldots, t_j$ are the subtrees of $t$ whose roots are the children of the root of $t$ and the root of $t$ has size $i$.

\medskip

We can build an algebra out of any rooted tree class, $\mathcal{T}$, by simply taking the polynomial algebra generated by the elements of the class.  Viewing monomials of trees as disjoint unions of trees we can also view this algebra as the vector space spanned by all forests of trees from $\mathcal{T}$ with disjoint union as multiplication.  Note that even if $\mathcal{T}$ is a simple class of trees we have not imposed a $\Phi$-structure on the forests -- in cases where the $\Phi$-structure has a natural algebraic interpretation this can be done, for example plane trees would correspond to the noncommutative polynomial algebra to preserve the order structure on forests, see for example \cite{Fo08}.

Suppose now that $\mathcal{T}$ is a simple class of trees.  Then every subtree of a tree in the class is also in the class, and so we can define a bialgebra structure using the following coproduct
\[
\Delta(f) = \sum_W f_W \otimes (f\setminus f_W)
\]
where the sum runs over subsets of vertices of the forest $f$ with no two vertices descendants one of the other.
The counit in this case is the algebra homomorphism which maps the empty tree to $1$ and maps every nonempty tree to $0$.  This gives a graded connected bialgebra and hence a Hopf algebra.   In the case where $\mathcal{T}$ is the class of all rooted trees, either decorated or undecorated, this is known as the Connes-Kreimer Hopf algebra of rooted trees and we will denote it $H_{\mathcal{R}}$ in the undecorated case and $H_{\mathcal{R}'}$ in the decorated case.

Define $B_+ \in \text{End}(H_\mathcal{R})$ such that for trees, $t_1, \ldots, t_n$, $B_+(t_1\cdots t_n)$ is the tree whose root is adjacent to the trees $t_1, \ldots, t_n$.  Tor each decoration $c\in \mathbb{Z}_{\geq 0}$ we can define $B_+ \in \text{End}(H_\mathcal{R}')$ such that $B_+(t_1\cdots t_n)$ is the tree constructed in the same way with root given size $c$.

\section{The unified result}

In this section we will discuss the results of each community and give an encapsulating result.  To begin with we need some definitions and notation.

\subsection{Preliminaries}
We say a map, $B$, from a class of forests, $\mathcal{T}$, to a field, $\mathbb{K}$, is a \emph{hook weight} if there exist $B_n \in \mathbb{K}$ such that for all forests $f \in \mathcal{T}$: $B(f) = \prod_{v \in V(f)}B_{\left|f_v\right|}$.

We can take the power series of a hook weight applied to each tree in a class of trees. We call this power series the \emph{hook length series} of $\mathcal{T}$ with respect to $B$ and denote it by \[F_{\mathcal{T}, B}(z) = \sum_{t \in \mathcal{T}}B(t)z^{|t|}.\]

For a power series $\phi(x) = \sum_{n \geq 0}\phi_nx^n$, we call, $F_{\phi, B}(z)$ given by 
\[F_{\phi, B}(z) = \sum_{t \in \mathcal{O}'}w_{\phi}(t)B(t)z^{|t|},\]
the \emph{hook length series} of $\phi$. Here $\mathcal{O} \cong \mathcal{Z} \times \textsc{seq}(\mathcal{O})$ is the class of plane trees and $w_{\phi(t)} = \prod_{v \in t}\phi_{\text{deg}(v)}$.

We can also define hook weights for decorated forests. A map, $B: \mathcal{T}' \to \mathbb{K}$, is called a \emph{hook weight} of the class of decorated forests, $\mathcal{T}'$, if there exist $B_n \in \mathbb{K}$ such that for all decorated forests, $f \in \mathcal{T}'$: $B(f) = \prod_{v \in V(f)}B_{\left|f_v\right|}$. This definition is the same as for ordinary trees except that the size of a decorated forest is the sum of the size of its vertices instead of the number of vertices it has.

Thus we can also define the \emph{hook length series} for a class of decorated trees, $\mathcal{T}'$:
\[F_{\mathcal{T}', B}(z) = \sum_{t \in \mathcal{T}'}B(t)z^{|t|}\]
and for a power series  $\varphi(z, x) = \sum_{n,m \geq 0}\phi_{m,n}z^mx^n$:
\[F_{\varphi, B}(z) = \sum_{t \in \mathcal{O}'}w_{\varphi}(t)B(t)z^{|t|}\]
where $\mathcal{O}' \cong \textsc{seq}(\mathcal{Z}) \times \textsc{seq}(\mathcal{O}')$ is the class of decorated plane trees and $w_{\varphi(t)} = \linebreak \prod_{v \in t}\varphi_{|v|-1, \text{deg}(v)}$.

\medskip

The B-series and quantum field theory communities take a more functional approach and so we need the following additional definitions. For a hook weight, $B$, let $L_B \in \text{End}(\mathbb{K}[z])$ be the map defined by $L_B(z^n) = B_{n+1}z^{n+1}$ and extended linearly. Similarly, define $L_B^*$ to be the linear operator such that $L_B^*(z^n) = (n+1)B_{n+1}z^n$ for all $n \in \mathbb{N}$.

Define $\theta$ to be the operator that takes $f(z)$ to $\theta(f)(z) = z\frac{d}{dz}f(z)$. A useful property of this operator is that $P(\theta)(z^n)=P(n)z^n$ for all polynomials, $P(x) \in \mathbb{K}[x]$.

\subsection{The main result}

Now we are ready to give the main result, first without decorations and then with decorations, as found in the first author's MSc thesis \cite{Jo14}.

\begin{thm}
\label{thm with main result}
Let $\phi(x)$ be a formal power series and $B$ be a hook weight.  Then we have the following three properties.
\begin{enumerate}
\item $F_{\phi, B}$ satisfies the recurrence:
\[[z^n]F_{\phi, B}(z) = B_n[z^{n-1}]\phi(F_{\phi, B}(z)),\quad{} \forall k \geq 1.\]
\item $F_{\phi, B}$ satisfies:
\[F_{\phi, B}(z) = L_B(\phi(F_{\phi, B}(z))).\]
\item  $F_{\phi, B}$ is a solution to the differential equation:
\[F_{\phi, B}'(z) = L_B^*(1+\theta)(\phi(F_{\phi, B}(z))).\]
\end{enumerate}
\end{thm}

The first item of the theorem is Kuba and Panholzer's result \cite{KP12}.  The second item is Mazza's result \cite{Ma04}.  The combinatorial community also had this differential equation formulation in certain cases \cite{LV86,BFS92,CGG09,Se05,GS06,KP14}.
Panzer \cite{Pa11} had some particular cases in the form of the third item.  He also looked at decorated trees as they are very natural in the renormalization Hopf algebra context.  In the decorated language we can give a simultaneous generalization of all of these results as follows.

\begin{thm}
\label{thm with main result decorated}
Let $\varphi(z, x)$ be a bivariate formal power series and $B$ be a hook weight.  Then we have the following three properties.
\begin{enumerate}
\item $F_{\varphi, B}$ satisfies the recurrence:
\[[z^n]F_{\varphi, B}(z) = B_n[z^{n-1}]\varphi(z, F_{\varphi, B}(z)),\quad{} \forall k \geq 1.\]
\item $F_{\varphi, B}$ satisfies:
\[F_{\varphi, B}(z) = L_B(\varphi(z, F_{\varphi, B}(z))).\]
\item  $F_{\varphi, B}$ is a solution to the differential equation:
\[F_{\varphi, B}'(z) = L_B^*(1+\theta)(\varphi(z, F_{\varphi, B}(z))).\]
\end{enumerate}
\end{thm}

\begin{proof}
We will first prove the recurrence by induction on $n$.

For $n = 1$, \[[z^1]F_{\varphi, B}(z) = w_\varphi(\bullet)B(\bullet) = \varphi_{0,0}B_1 = B_1[z^0]\varphi(z, F_{\varphi, B}(z)).\]

For $n > 1$,
\begin{align*}
[z^n]F_{\varphi, B}(z) &= \sum_{t \in \mathcal{O}'_n}w_\varphi(t)B(t) \\
 &= \sum_{i = 1}^n\sum_{j \geq 1}\varphi_{i,j}B_n\sum_{\stackrel{n_1 + \cdots + n_j = n-i-1}{n_1, \ldots, n_j \geq 1}}\sum_{t_1 \in \mathcal{O}'_{n_1}, \ldots, t_j \in \mathcal{O}'_{n_j}}\prod_{l=1}^{j}\left(w_\varphi(t_l)B(t_l)\right) \\
 &= B_n\sum_{i = 1}^n\sum_{j \geq 1}\varphi_{i,j}\sum_{\stackrel{n_1 + \cdots + n_j = n-i-1}{n_1, \ldots, n_j \geq 1}}\prod_{l=1}^{j}\left(\sum_{t_l \in \mathcal{O}'_{n_l}}w_\varphi(t_l)B(t_l)\right) \\
 &= B_n\sum_{i = 1}^n\sum_{j \geq 1}\varphi_{i,j}\sum_{\stackrel{n_1 + \cdots + n_j = n-i-1}{n_1, \ldots, n_j \geq 1}}\prod_{l=1}^{j}[z^{n_l}]F_{\varphi, B}(z) \\
 &= B_n [z^{n-1}]\varphi(z, F_{\varphi, B}(z)).
\end{align*}

The other two properties are equivalent to the recurrence. This can be seen by applying coefficient extraction to the equations:
\[
	[z^{n}]F'(z)=(n+1)[z^{n+1}]F(z)
\]
and
\[
	[z^n]L_B^*(1+\theta)(\varphi(z, F_{\varphi, B}(z))) = (n+1)B_{n+1}[z^n]\varphi(z, F_{\varphi, B}(z))
\]
give that the differential equation is equivalent to the recurrence. Also,
\[
	[z^n]L_B(\varphi(z, F_{\varphi, B}(z))) = B_{n+1}[z^n]\varphi(z, F_{\varphi, B}(z))
\]
gives that the second property is equivalent to the recurrence.
\end{proof}

\section{Using the differential equation in enumerative combinatorics}\label{sec using Mazza}

The first item in the theorems does not explain why some choices of class and final form result in nice hook weights, or why certain classes and hook weights give a nice final form.  It also doesn't give any hint of where to look for novel hook length formulae of combinatorial interest which are not simple extensions of know results.  

The differential equation perspective gives a little bit of traction on these issues.  One way for a choice of hook weight and tree class to have a nice hook length formula is if the differential equation is solvable.  One place to look for novel hook length formulae is among appropriate differential equations.  We have a number of such new hook length formulae, all but one of which were first reported in the first author's MSc thesis \cite{Jo14} and are shown here in table~\ref{table of new hlf}.  Kuba and Panholzer also recently realized the value differential equations for finding interesting formulae in their new study on bilabelled trees (\cite{KP14} section 5). 

To be able to deal with hook weights which do not correspond nicely to continuous functions we first need to develop some tools.

\subsection{New tools for tree hook length formulae}

To use the differential equation we need to convert the hook weights, which are defined on the natural numbers, to functions defined on $\mathbb{R}_{> 0}$.  To extract useful information from the differential equation these functions need to stay as simple as possible.  
Some hook weights naturally correspond to piecewise functions and are best dealt with by breaking the tree specification up to match the pieces, as discussed in the first two methods below.  Others have an exponential dependence which is not well-behaved in the differential equation but can be dealt with by scaling as discussed in the third method below.

\subsubsection{Leafless method and system method}
			The first method we present is called the system method. The method is used when
			\[B_k =
			\begin{cases}
				B_k^{(1)} & \text{if } k \in P_{1} \\
				 & \vdots \\
				B_k^{(m)} & \text{if } k \in P_{m} \\
			\end{cases}
			\]
			for some partition $P_{1} \cup \cdots \cup P_{m} = \mathbb{N}^+$ and $L_{B^{(i)}}^* \neq L_{B^{(j)}}$ for all $i \neq j$.
			
			For a combinatorial class, $\mathcal{C}$, and set, $S \in \mathbb{N}$, let $\mathcal{C}_S = \bigcup_{n \in S} \mathcal{C}_n$.
			
			Given a simple tree class $\mathcal{T}$ suppose that it can be easily separated into classes: $\mathcal{T}_{P_{1}}, \ldots, \mathcal{T}_{P_m}$ with each class satisfying some psuedo-simple relation: \[\mathcal{T}_{P_i} \cong \mathcal{Z} \times \Phi_i(\mathcal{T}_{P_1}, \ldots, \mathcal{T}_{P_m}).\]
			
			From theorem \ref{thm with main result decorated}, each $F_{\mathcal{T}_{P_i}, B^{(i)}}(z)$ satisfies a differential equation of the form:
			\[F_{\mathcal{T}_{P_i}, B^{(i)}}'(z) = L_{B^{(i)}}^*(1+\theta)\phi_i(F_{\mathcal{T}_{P_1}, B^{(1)}}(z), F_{\mathcal{T}_{P_m}, B^{(m)}}(z)).\]
			By adding the solutions to the system of differential equations together we obtain $F_{\mathcal{T}, B}(z)$.
			
			Splitting $\mathcal{T}$ into such and using the system of differential equations to prove a hook length formula is called the \emph{system method}.
			
			\medskip{}
			
                        In the special case where 
      			\[B_k=
				\begin{cases}
					a & \text{ if } k=1 \\
					g(k) & \text{ if } k > 1 \\
				\end{cases}
			\]
			for some function $g$ with $g(1)$ undefined or $g(1) \neq a$, we get 
the system $\mathcal{T}_1 \cong \phi_0 \mathcal{Z}$ and $\mathcal{T}_{> 1} \cong \mathbb{Z} \times (\Phi(\mathcal{T}) - \phi_0)$.
In this case the system method is called the \emph{leafless method}.  It is so called because it essentially removes the leaves of the trees to produce the differential equation. An example of how to use the leafless method can be found in example \ref{plane tree ex}.


\subsubsection{Scaled method}
			The final method we shall present is called the scaled method. The scaled method is used when
			\[B_k = r^{k-1}C_k\]
			for some other hook weight $C$. In this case:
			\[L_B^*(x) = \sum_{i=0}^\infty \sum_{j=0}^{i}\binom{m}{i}\frac{(\ln r)^i}{ri!}c_{m-i}x^m,\]
			which does not give an easily solvable differential equation.

			We can bypass this by the following observation
                        \begin{prop}
                          Suppose $B$ and $C$ are hook weights satisfying $B_k = r^{k-1}C_k$ for some $r\in \mathbb{K}$.  Then $L_B(p(z)) = L_C(p(rz))$ for all $p(z)\in \mathbb{K}[z]$.
                        \end{prop}

                        \begin{proof}
                          For $n \in \mathbb{N}$ we have
                          \[
                          	L_C((rz)^{n}) = C_{n+1}r^nz^{n+1} 
                          	= B_{n+1}z^{n+1} 
                          	= L_B(z^n).
                          \]
                          Therefore, by linearity, $L_B(p(z)) = L_C(p(rz))$ for all $p(z)\in \mathbb{K}[z]$.
                        \end{proof}
                        
			The proposition implies that
			\[[z^n]L_B^*(1+\theta)x(z) = [z^n]L_C^*(1+\theta)x(rz).\]
			
			Therefore for the simple tree class, $\mathcal{T}$, $F_{\mathcal{T}, B}$ solves the differential equation:
			\[F_{\mathcal{T}, B}'(z) = L_C^*(1+\theta)\phi(F_{\mathcal{T}, B}(rz)).\]
			Using this differential equation to prove a hook length formula is called the \emph{scaled method}			
			We call this method the scaled method because we scale the hook length series by $r$.

\subsection{New hook length formulae}

In the first author's MSc thesis \cite{Jo14}, he collected known hook length formula into a catalogue of tables. Here we present a portion of that table, in table \ref{table of new hlf}, that includes the hook length formula discovered and proved by the first author.  The method column of table \ref{table of new hlf} explains how the formula was found and can be proved: by the Kuba-Panholzer recurrence (KP), by the Mazza differential equation (DE), by the leafless method (Leaf), by the system method (Sys), by the scaled method (Sc) or some combination as indicated.

The second formula of the table is new to this paper and is discussed in more detail in example~\ref{new eg}.

\begin{table}
\caption{Table of new hook length formulae}
\label{table of new hlf}

$\begin{array}{|c|l|l|c|}
\hline
\phi & \multicolumn{1}{|c|}{B} & \multicolumn{1}{|c|}{F} & \text{method} \\
\hline
(1+x)^2 &  B_k =  \frac{((k-1)!)^2}{(2k-1)!} & F_n = \frac{2^n}{(n+1)!n!} & \text{KP} \\
(1+x)^2 & B_k=k & F_n=\text{A699}_{n+1} & \text{KP} \\
\hline
1+x+x^2 & B_k = 
\begin{cases}{}
	1 & \text{if } k=1 \\
	\frac{1}{k-1} & \text{if } k > 1
\end{cases}
&  F(z) = \frac{z}{1-z} &  \text{DE} \\
\hline
(1+x)^r &  B_k = \frac{1}{\binom{k+r-2}{r-1}} &  F(z)=\frac{z}{1-z} &  \text{KP} \\
\hline
1+x^r &  B_k =
\begin{cases}
	1 & \text{if } k = 1 \\
	\frac{(r-1)!r^{r-1}}{\prod_{i=0}^{r-2}(k+ir-1)} & \text{if } k > 1
\end{cases}
&  F(z)=\frac{z}{1-z^r} &  \text{KP} \\
\hline
\frac{1}{1-x} &  B_k = 
\begin{cases}
	1 & \text{if } k = 1 \\
	2^{2-k} & \text{if }  k > 1
\end{cases}
&  F(z)=\frac{z}{1-z} &  \text{Sc+Leaf} \\
\frac{1}{1-x} & B_k =
	\begin{cases}
	1 & \text{if } k \leq 2 \\
	\frac{(\sqrt{-1})^{k-1}}{2} & \text{if $k > 2$ is odd} \\
	\frac{(\sqrt{-1})^{k-2}}{2} & \text{if $k > 2$ is even}
\end{cases}
& F(z) = \frac{z+z^2}{1+z^2} & \text{Sc+Sys} \\
\hline
\frac{1}{(1-x)^2} & B_k = \frac{1}{k} & F(z) = 1-(1-3z)^{\frac{1}{3}} & \text{DE} \\
\frac{1}{(1-x)^2} & B_k = 
\begin{cases}
	1 & \text{if } k=1 \\
	\frac{1}{2^{k-3}(k+2)} & \text{if } k > 1
\end{cases}
& F(z) = \frac{z}{1-z} & \text{Sc+Leaf} \\
\hline
e^x &  B_k = 1+\frac{1}{k} &  F(z)= {-2}\log\left(\frac{1+\sqrt{1-4z}}{2}\right) &  \text{DE} \\
e^x &  B_k = \frac{1}{k\binom{k+a-2}{a-1}} &  F(z)=a\log\left(\frac{a}{a-z}\right) &  \text{KP} \\
e^x &  B_k = \frac{2}{k}-1 &  F(z)=\log\left(z+\sqrt{1+z^2}\right) & \text{DE} \\
\hline
\cosh(x) &  B_k=\frac{2}{k}-1 &  F(z)=\log\left(z+\sqrt{1+z^2}\right) &  \text{DE} \\
\hline
1 - \log(1-x) &  B_k=
\begin{cases}
	1 & \text{if } k=1 \\
	\frac{1}{k}-1 & \text{if } k > 1
\end{cases}
&  F(z)=1 + z - \sqrt{1+z^2} &  \text{Leaf} \\
1 - \log(1-x) &  B_k =
\begin{cases}
	1 & \text{if } k=1 \\
	k-1 & \text{if } k > 1
\end{cases}
&  F(z)=1-\frac{1}{\sum_{n\geq 0}n!z^n} &  \text{Leaf} \\
1 - \log(1-x)  &  B_k = 
\begin{cases}
	1 & \text{if } k=1 \\
	\frac{k-1}{2^{k-1}-1} & \text{if } k > 1
\end{cases}
&  F(z)=\frac{z}{1-z} &  \text{Sc+Leaf} \\
\hline
\end{array}$
\end{table}

\begin{example}\label{plane tree ex}
	This example illustrates how to use two of the new methods, the leafless method and the scaled method, in conjunction to prove a hook length formula.
	
	Let $B_k= \frac{2}{2^{k-1}}$ if $k > 1$ and $B_1 = 1$. Also let $\phi(x) = \frac{1}{1-x}$; this $\phi$ encodes the class of plane trees. Finally let $F = F_{\phi, B}$.
	
	Since $B_k$ contains a factor of $\frac{1}{2^{k-1}}$ we shall use the scaled method. Let $C_k=2$ then $B_k=C_k\frac{1}{2^{k-1}}$. By the scaled method \[L_B^*(1+\theta)(x(z)) = L_C^*(1+\theta)\left(x\left(z/2\right)\right) = 2x\left(z/2\right)+2z\frac{d}{dz}x\left(z/2\right),\] because $L_C^*(n)=2n$.
	
	Since $2^{2-1} = 2 \neq 1 = B_1$, to use the differential equation to solve this hook length formula we need to use the leafless method. By the leafless method, $F(z)$ solves the differential equation
	\[F'(z) - \phi_0B_1 = L_B^*(1+\theta)(\phi(F(z))).\]
	
	Putting these two methods together we get that $F(z)$ solves the differential equation:
	\[F'(z) - 1 = \frac{2}{1-F(z/2)}+\frac{F'(z/2)}{(1-F(z/2))^2}.\]
	Plugging in $F(z)=\frac{z}{1-z}$ we can see that the differential equation is satisfied.

	Therefore, $F_{\phi,B}(z)=\frac{z}{1-z}$.
\end{example}


\begin{example}\label{new eg}
	Let $B_k = k$ and $\phi(x) = (1+x)^2$ then $F(z) = F_{\phi, B}(z)$ satisfies the recurrence \[[z^n]F(z)=n[z^{n-1}](1+F(z))^2.\] This recurrence is similar to the recurrence \[S_n = (n-1)\sum_{j=1}^{n-1}S_jS_{n-j}\] from \cite{St77}. The $S_n$ here count the number of irreducible arc diagrams. The OEIS number \cite{OEIS} of this sequence is A699. If we consider $S(z) = \sum_{n \geq 1}S_{n+1}z^n$ then
	\begin{align*}
		[z^n]S(z) &= S_{n+1} \\
		&= n\sum_{j=1}^nS_jS_{n+1-j} \\
		&= n[z^{n-1}](S_1+S(z))^2.
	\end{align*}
	Since $S_1= 1$, we have that $S(z) = F(z)$.
\end{example}

\section{Using tree hook length formulae in quantum field theory}\label{sec qft models}

Combinatorial Dyson-Schwinger equations are functional equations with solutions in \linebreak $H_\mathcal{R}[[z]]$ using grafting operators, products, inverses, and the empty tree, $\mathbb{I}$.  
As an
example consider
\[
  X(z) = \mathbb{I} - zB_+(X(z)^{-1})
\]
where the inverted series should be expanded as a geometric series.  This has as a solution
\[
X(z) = \mathbb{I} - z\vcenter{\hbox{\includegraphics{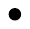}}} - z^2\vcenter{\hbox{\includegraphics{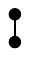}}} - z^3\left(\vcenter{\hbox{\includegraphics{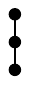}}} + \vcenter{\hbox{\includegraphics{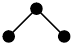}}}\right) - z^4\left(\vcenter{\hbox{\includegraphics{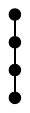}}} + \vcenter{\hbox{\includegraphics{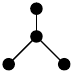}}} + 2\vcenter{\hbox{\includegraphics{treefacteg}}} + \vcenter{\hbox{\includegraphics{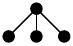}}}\right) + \cdots .
\]
It is possible to have more than one $B_+$ appearing and to have linked systems of equations.  The precise definition of what forms are allowed depends on the context at hand, compare for example \cite{Fo08} and \cite{Ye08}, but in all cases they act as specifications for tree classes:  $B_+$ plays the role of $\mathcal{Z}\times$, inverses play the role of \textsc{seq}.  The series which arise as solutions to combinatorial Dyson-Schwinger equations are intermediates between the combinatorial classes themselves and usual generating functions.  Linear combinations of trees appear as coefficients, not just the number of trees, but since these are trees in $H_\mathcal{R}$ they appear with coefficients reflecting how many trees of a given shape appear in the specified class.

To convert combinatorial Dyson-Schwinger equations into analytic Dyson-Schwinger equations, which are the honest-to-goodness Dyson-Schwinger equations of physics, we need simply to apply Feynman rules.  In the physical situation Feynman rules are rules which convert Feynman graphs into integrals with each edge and vertex of the graph contributing a factor to the integrand.  These integrals are divergent in interesting cases and need to be renormalized.  The end result is a function of various physical parameters, such as the momenta of the particles coming in and out of the process. 

\medskip

In our situation, we have series in trees and we will follow Panzer (\cite{Pa11} p38) by defining Feynman rules as algebra morphisms, $\phi:H_\mathcal{R} \rightarrow A$, to some commutative algebra, $A$, satisfying $\phi \circ  B_+ = L \circ \phi$
for some $L \in \text{End}(A)$.  In the physics case $A$ would be some algebra of functions of the parameters.

Hook weights provide a family of simple examples of such Feynman rules.  Panzer showed (\cite{Pa11} Theorem 2.4.6) that for any $L\in \text{End}(A)$ there exists a unique morphism of unital algebras ${}^L\rho:H_{\mathcal{R}}\rightarrow A$ such that ${}^L\rho\circ B_+ = L\circ {}^L\rho$ and analogously for the decorated case (\cite{Pa11} section 2.5).
Specifically for hook weights we get
\begin{prop}
	Let $B$ be a hook weight. For any forest, $f$, we have
	\[{}^{L_B}\rho(f) = B(f)z^{|f|}.\]
\end{prop}

Hook length Feynman rules are particularly simple in that their one parameter comes with power the size of the forest and thus carries no new information compared to the counting variable we already had.  Thus in this very simple case we can conflate the two parameters and so for us $A$ will be $\mathbb{K}[z]$ and our Green functions will be single variable functions.

In still simplified but more physically realistic cases the Feynman rules would give polynomials in a parameter, call it $T$, of degree the size of the tree.  Then the Green functions would be functions of both $T$ and $z$ where any monomial in the expansion of the Green function would have degree in $z$ at least as large as the degree in $T$.  The part with the same degree in $z$ and $T$ is known in quantum field theory calculations as the \emph{leading log} part.  Thus another interpretation of the very special case we are discussing here is as the leading log part of a more complicated set up.

\medskip

Returning to hook length Feynman rules, consider what the hook length formulae mean in this context. 
 Hook length formulae tell us that particular choices of combinatorial class and hook weight give nice series, say with a closed form or a nice combinatorial interpretation.  In quantum field theory language hook length formulae give us particular choices of combinatorial Dyson-Schwinger equation and Feynman rules so that the analytic Dyson-Schwinger equation has a closed form, or otherwise is combinatorially nice.

\begin{example}
  For example, take Postnikov's formula, given in the introduction as \eqref{eq of Postnikov}.  Here we are looking at binary trees, so the Dyson-Schwinger equation is $X(z) = \mathbb{I} + zB_+(X(z)^2)$.  The hook weight is given by $B_k = 1 + \frac{1}{k}$, so the Feynman rules, $\phi$, are 
  \[
  t\mapsto \prod_{v\in t} \left(1+\frac{1}{|t_v|}\right)z^{|t|},
  \]
or equivalently
\[
\phi(B_+(f)) = z\left(1+ \frac{1}{|f|}\right)\phi(f).
\]
Then Postnikov's formula tells us that $F_n = 2^n(n+1)^{n-1}$.  Taking the exponential generating function we get $F(z) = \frac{-W(-2z)}{2z}-1$ where $W$ is the Lambert $W$-function  (see \cite{Jo14} p12 for details).  So the Green function $G(z)=1+F(z)$ comes from a series expansion of the Lambert $W$-function.
\end{example}

\begin{example}
  As another example consider the eighth formula in table \ref{table of new hlf}.  This is a nice example because it is new and it uses the usual inverse of tree factorial Feynman rules.  The tree specification in this case is $\mathcal{T} = \mathcal{Z}\times \textsc{seq}(\mathcal{T})^2$ so the Dyson-Schwinger equation is
\[
X(z) = \mathbb{I} - zB_+\left(\frac{1}{X(z)^2}\right).  
\]
Then the hook length formula tells us that the Green function is 
\[
G(z) = 1-F(z) = (1-3z)^{\frac{1}{3}},
\]
which is a nice closed form.
\end{example}

\medskip

A final point concerns the coalgebra structure.  $B_+$ is a Hochschild 1-cocycle \cite{CK98}, specifically
\[
\Delta B_+ = (\text{id}\otimes B_+)\Delta + B_+ \otimes \mathbb{I}
\]
and similarly for each $B_+^a$ in the decorated case.
Panzer shows (\cite{Pa11} Theorem 2.4.6) that if $A$ is a bialgebra and $\Delta L = (\text{id} \otimes L)\Delta + L\otimes 1$ then ${}^L\rho$ is a bialgebra homomorphism, and if further $A$ is a Hopf algebra then ${}^L\rho$ is a Hopf algebra homomorphism.  The analogous statement holds in the decorated case (\cite{Pa11} section 2.5).

Unfortunately, only multiples of the inverse of tree factorial are 1-cocycles.
\begin{prop}
  Let $B$ be a hook weight.  Define $\Delta$ on $\mathbb{K}[z]$ by extending $\Delta(z) = 1\otimes z + z\otimes 1$ as an algebra homomorphism.  If
\[
L_B \Delta = (\text{id}\otimes L_B)\Delta + L_B \otimes 1
\]
then $B_n = c/n$ for some $c\in \mathbb{K}$. 
\end{prop}

\begin{proof}
  \begin{align*}
    B_{n+1} \sum_{i=0}^{n+1}\binom{n+1}{i} z^i\otimes z^{n+1-i}  
    = L_B \Delta (z^n) & = (\text{id}\otimes L_B)\Delta + L_B \otimes 1 \\
    & = \sum_{i=0}^n\binom{n}{i} B_{n-i+1}z^i \otimes z^{n-i+1} + B_{n+1}z^{n+1}\otimes 1 
  \end{align*}
  Equating coefficients we get 
  \[
  B_{n+1}\binom{n+1}{i} = B_{n-i}\binom{n}{i}
  \]
  for all $0\leq i \leq n$ and so in particular with $i=n$, $B_{n+1}(n+1) = B_1$ giving the result. 
\end{proof}

Therefore the inverse of tree factorial is special in this quantum field theory context.  None-the-less there are many interesting examples using the inverse of tree factorial (see \cite{Jo14} chapter 6 for a comprehensive listing) and the other hook weights are at least buildable by $B_+$ -- that is they are Feynman rules by the present definition.

\bibliographystyle{plain}
\bibliography{main}

\end{document}